\documentclass[12pt,a4paper]{amsart}

\input xy
\xyoption{all}

\DeclareMathAlphabet{\mathpzc}{OT1}{pzc}{m}{it}

\usepackage{amsfonts,amsmath,amssymb,indentfirst,mathrsfs,amscd}
\usepackage[mathscr]{eucal}
\usepackage[active]{srcltx}
\usepackage{graphicx}

\usepackage[dvips]{color}
\usepackage{color}

\DeclareMathOperator{\rk}{rk \,}          

\DeclareMathOperator{\End}{End\,}           
\DeclareMathOperator{\Id}{Id\,}             
\DeclareMathOperator{\U}{U}

\DeclareMathOperator{\SU}{SU}
\DeclareMathOperator{\Spin}{Spin}

\DeclareMathOperator{\Sp}{Sp}
\DeclareMathOperator{\SO}{SO}

\DeclareMathOperator{\Tr}{Tr\,}       

\newtheorem{thm}{Theorem}
\newtheorem{prop}[thm]{Proposition}
\newtheorem{lemma}[thm]{Lemma}
\newtheorem{cor}[thm]{Corollary}

\theoremstyle{definition}
\newtheorem{defn}[thm]{Definition}

\theoremstyle{remark}
\newtheorem{rmk}[thm]{Remark}

\theoremstyle{remark}

\newcommand{\vol}{\mathrm{vol}}

\newcommand{\bd}{\partial}
\newcommand{\bbd}{\bar\partial}
\newcommand{\ox}{\otimes}
\newcommand{\la}{\langle}
\newcommand{\ra}{\rangle}

\newcommand{\cF}{\mathcal{F}}
\newcommand{\cI}{\mathcal{I}}
\newcommand{\cL}{\mathcal{L}}
\newcommand{\cU}{\mathcal{U}} 

\newcommand{\x}{\times}

\newcommand{\CC}{\mathbb{C}} 
\newcommand{\QQ}{\mathbb{Q}} 
\newcommand{\HH}{\mathbb{H}} 
\newcommand{\RR}{\mathbb{R}} 
\newcommand{\ZZ}{\mathbb{Z}} 

\newcommand{\too}{\longrightarrow}
\newcommand{\imat}{\sqrt{-1}} 
\newcommand{\CP}{\CC P}
\renewcommand{\Re}{\mathrm{Re}}

\author[V. Mu\~{n}oz]{Vicente Mu\~{n}oz}

\begin{document}

\title{On rotation of complex structures}

\thanks{This work has been partially supported by (Spanish) MICINN Project MTM2010-17389.}

\begin{abstract}
 We put in a general framework the situations in which a Riemannian manifold admits a 
family of compatible complex structures, including hyperk\"ahler metrics and the
$\Spin$-rotations of \cite{Mu}. We determine the (polystable) holomorphic bundles which are rotable, i.e., they
remain holomorphic when we change a complex structure by a different one in the family.

\vspace{1cm}

\noindent {Email: {\tt vicente.munoz@mat.ucm.es}}

\noindent Tel: +34 913944464 

\noindent Fax: +34 913944564

\noindent {Address: Facultad de Matem\'aticas, Universidad Complutense de Madrid,
Plaza Ciencias 3, 28040 Madrid, Spain}
\end{abstract}

\maketitle

\section{Introduction}

Hyperk\"ahler manifolds admit an $S^2$ family  
of complex structures, all of them integrable and compatible
with the metric. This produces a collection of different complex manifolds, all of them
naturally related, but very often with different algebro-geometric properties. For example, 
it is typical that some of the manifolds in the family are algebraic and others are not. Other 
properties, like the Hodge structures, also change in the family. 

There are other situations in which a Riemannian manifold admits a family of compatible complex
structures, like the $\SU(4)$-structures compatible with a $\Spin(7)$-structure on the $8$-torus,
studied in \cite{Mu}. This consists of an $S^6$ family of complex structures, that is,
a family of complex $4$-tori all of them naturally related, and again with very different
algebro-geometric properties. Indeed, in \cite{Mu} there is an example of an abelian
variety $X$ with $\End(X)=\QQ[\sqrt{-d}] \x\QQ[\sqrt{-d}]$, $d\in \ZZ_{>0}$ square-free, 
and another abelian
variety $X'$ in the same family with $\End(X') 
=\QQ[\sqrt{-d},\sqrt{e}]$, $d,e\in\ZZ_{>0}$ square-free.
Also, it is typical that some of the complex $4$-tori in the family are algebraic whereas 
others are not.

In the present note, we aim to put both previous examples in a general framework. Moreover, we shall
describe other instances of the same phenomena, like the case of the product of two K3 surfaces.

\medskip

Let $E\to M$ be a (hermitian) complex vector bundle over a K\"ahler manifold $(M,\omega)$. Then $E$ admits a
Hermitian-Yang-Mills connection (HYM connection, for short) if there is a hermitian connection $A$
such that  
 $$
 \left\{ \begin{array}{l} F_A \in \bigwedge\nolimits^{1,1} (\End E) \\
  \Lambda F_A= \lambda \, \Id \end{array} \right.
 $$
for a constant $\lambda$, where $\Lambda:\bigwedge^2 \to \bigwedge^0$ denotes contraction with $\omega$.
Decomposing $A=\bd_A+\bbd_A$ into $(1,0)$ and $(0,1)$-components, we have that $\bbd_A$ is a holomorphic
structure on $E$ and moreover $(E,\bbd_A)$ is a polystable bundle with respect to $\omega$ (a direct sum
of stable bundles all of the same slope). The reciprocal also holds: a polystable bundle with respect
to $\omega$ admits a HYM connection. This is the content of the Hitchin-Kobayashi correspondence \cite{UY}.

If $M$ admits a family of complex structures compatible with the given metric, then 
$E\to M$ might be HYM with respect to all (or a subfamily) of the K\"ahler structures simultaneously.
In the case of hyperk\"ahler manifolds, such bundles are called hyperholomorphic and
have been extensively studied by Verbitsky \cite{Verbitsky}. In the case of complex $4$-dimensional
tori with $\Spin(7)$-structures, such bundles have been described in \cite{Mu}, where they
are called $\Spin$-rotable bundles. 

A bundle $E$ which is HYM with respect to different complex structures in one of these families 
is an interesting object, since it determines holomorphic bundles for different complex structures
on the given (smooth) manifold $M$. Here, we shall called such bundles \emph{rotable}.
In particular, the Chern classes of a rotable bundle $E$ 
are algebraic cycles on $(M,J)$ for any of these complex structure $J$ such that 
$(M,J)$ is an algebraic complex manifold. This is an indirect route for constructing algebraic cycles.
If this happens, we shall say that $c_j(E)$ are \emph{rotable algebraic cycles}.

Another instance in which rotations of complex structures have been used is \cite{Sc}.
Schlickewei has used this mechanism to determine Hodge classes in self-products of K3 surfaces
which are rotable algebraic cycles, thereby proving the Hodge conjecture in some cases.

We will describe the bundles which are HYM with respect to a
family of complex structures compatible with a Riemannian structure $(M,g)$
in the different situations of rotations of complex structures that we analyse.

\noindent \textbf{Acknowledgements.} I am grateful to Ivan Smith for a
 kind invitation to Cambridge University to discuss about these matters.
The question about the study of rotations of complex structures for a
product of two K3 surfaces was prompted to the author by Ivan Smith and
Richard Thomas. I would like to thank Misha Verbitsky and Daniel Huybrechts
for useful conversations. Finally, many thanks to the referee for very 
helpful comments.

\section{Rotation of complex structures}

Let $M$ be a Riemannian manifold of real dimension $2n$, and let $H<\SO(2n)$ be its
holonomy group. Consider a second group $G$ such that 
 $$
 H< G<\SO(2n).
 $$
Here $G$ has the role of a ``ground'' group, that is, we fix the $G$-structure of $M$.
So if $G=\SO(2n)$, we are merely fixing the Riemannian structure of $M$.

A compatible complex structure is a reduction (parallel with respect to
the Levi-Civita connection) to a group $U\cong \U(n)$ with $H<U<G$.
This is equivalent to give a K\"ahler structure on $M$. We see this
as follows: fix a base-point $p\in M$ and trivialize $T_pM=\RR^{2n}$.
A tensor $T_p$ on $T_pM$ determines a parallel tensor $T$ on $M$ by 
doing parallel transport along curves, if and only if it is fixed by $H$.
A complex structure on $T_pM$ is detemined by $J_p:T_pM\to T_pM$
with $J_p^2=-\Id$, which is equivalent to giving a subgroup 
$U<\SO(2n)=\SO(T_pM)$, where $U\cong \U(n)$ are the elements which fix $J_p$.
Then $J_p$ determines $J$ with $\nabla J=0$ (that is, an integrable
complex structure) if and only if $H<U$.

We also consider the case of groups $U\cong \SU(n)$ with $H<U<G$ under
the same terminology, although in this case $M$ is endowed with a
K\"ahler structure $I$ \emph{plus} a parallel form $\theta$ of type $(n,0)$ 
with respect to $I$.

The set of compatible complex structures is thus
 $$
 \cU=\{ U \, | \, H<U<G \}.
 $$
Changing a complex structure $U_1\in \cU$ to another one $U_2\in \cU$ will be called 
a \emph{rotation of complex structures}.

We fix $U_0 \in\cU$ and consider 
 $$ 
 N=\{g\in G \, | \, g H g^{-1} =H\}_o
 $$
and 
 $$
 C=\{ g\in N \, | \, g \, U_0 \, g^{-1} =U_0\}_o\, ,
 $$
where the subindex \ $o$ \ means ``connected component of the identity''.
Clearly 
 $$
 H<C<N<G.
 $$ 
Conjugating $U_0$ via $g$ produces another complex
structure $U_g=g\, U_0\,  g^{-1}$. These complex structures are parametrized by 
 $$
 \cU'=N/C.
 $$
Note that $\cU'\subset \cU$. In the situations of this paper, these sets 
are equal.

Now we will analyze different instances of rotations of complex structures.

\section{Hyperk\"ahler rotations} \label{sec:1}

\subsection{K3 surfaces}
K3 surfaces are K\"ahler surfaces with holonomy $H=\SU(2)=\Sp(1)<G=\SO(4)$. In particular
K3 surfaces are hyperk\"ahler. 

The universal cover of $\SO(4)$ is $\widetilde{\SO(4)}=\SU(2)_L\x \SU(2)_R$,
where $\SU(2)_L$ and $\SU(2)_R$ are two copies of $\SU(2)=\Sp(1)$. If we
consider $\RR^4$ as the space of quaternions $\HH$, then $\SU(2)_L$ acts as 
the unit quaternions $\Sp(1)=S^3\subset \HH$ by multiplication on the left, and
$\SU(2)_R$ acts by multiplication on the right.

The holonomy group of a K3 surface $M$ is $H=\SU(2)_L < \SO(4)$. There are
three complex structures $I,J,K$ and $\{L=aI+bJ+cK\, | \, a^2+b^2+c^2=1\}$  is
the family of all compatible complex structures on $M$. This family is a $2$-sphere.
Actually, the quaternions $i,j,k\in \Sp(1)=\SU(2)_R$, acting on the right on
$\HH=\RR^4$, produce the tensors $I,J,K:TM\to TM$, by parallel transport.

Now fix the complex structure $I$. This is the same as to consider the subgroup 
 $$
 U_I=\U(2)<\SO(4)
 $$ 
of all elements of $\SO(4)$ commuting with $I$. These are generated by
$\SU(2)_L$ and $S_I^1=
\{a \Id + \, b \, I\, | \, a^2+b^2=1\}\subset \SU(2)_R$. So $U_I=\SU(2)_L\cdot S_I^1$.
We have then
 \begin{align*}
  N &=\SO(4), \\
 C &=U_I.
 \end{align*}
The rotations of complex structures are given by 
 $$
  \cU'=\SO(4)/\U(2) .  
 $$
Note that $\cU'=\cU$ in this case. Also 
 $$
 \cU'=\SO(4)/(\SU(2)_L \cdot S_I^1) \cong \SU(2)_R/S_I^1 \cong S^2.
 $$ 
The action of $\SU(2)_R$ on
$\cU'$ is by conjugation, and it moves all $L=aI+bJ+cK$ transitively.

Using the metric, we write $\End (\RR^4) \cong (\RR^4)^*\ox (\RR^4)^*$. The
endomorphisms $aI+bJ+cK$, $(a,b,c)\in \RR^3$, correspond to antisymmetric
tensors, which are self-dual with respect to the Hodge $*$-operator, that
is, tensors in $\bigwedge^2_+$. Otherwise said, $\SO(4)$ acts on $\bigwedge^2_+$,
$I$ corresponds to the K\"ahler form $\omega_I$, the isotropy of $\omega_I$
is $U_I=\U(2)$, and $\SO(4)/\U(2)$ is the orbit of $\omega_I$ in $\bigwedge^2_+=\RR^3$.
This is the $2$-sphere $S(\bigwedge^2_+)$, i.e.,
 $$
  \SO(4)/\U(2) \cong S(\bigwedge\nolimits^2_+) = S^2,
 $$
naturally. The action of $\SO(4)/\SU(2)_L =\SU(2)_R/\pm \Id =\SO(3)$ is the standard action on
this $S^2$.

Suppose that $E\to M$ is a complex vector bundle with a connection which is HYM with
respect to $I$.
Then $F_A \in \bigwedge^{1,1}_I(\End E)$ and $\Lambda_I F_A =\lambda \Id$. 
We have a decomposition:
 $$ 
 \bigwedge\nolimits^2=\bigwedge\nolimits^2_+ \oplus \bigwedge\nolimits^2_-=
 \la \omega_I, \omega_J,\omega_K\ra \oplus
 \triangle^{1,1}_{I,prim}
 $$
(Here there is a slight abuse of notation: when refering to forms, $\bigwedge^r$ means the
bundle of $r$-forms on $M$; when dealing with a vector space $\RR^n$, $\bigwedge^r$ is
the $r$-th exterior power of $(\RR^n)^*$. This will happen throughout.)
From this it is clear that $\omega_J,\omega_K$ span the space $\triangle_I^{2,0}=
\Re (\bigwedge^{2,0}_I\oplus \bigwedge^{0,2}_I)$. Here $\triangle^{1,1}_I=\Re(\bigwedge^{1,1}_I)$
and $\triangle^{1,1}_{I,prim}$ is the space of primitive $(1,1)$-forms (those 
orthogonal to $\omega_I$).

There are two options:
\begin{itemize}
 \item If $\lambda=0$, then $F_A\in \bigwedge^{1,1}_{I, prim}(\End E)$, so $F_A\in
 \bigwedge^{1,1}_{L,prim}(\End E)$ for any $L\in \cU$. Then $E$ is HYM with respect
to all $L\in \cU$. Such bundle $E$ is called hyperholomorphic in the
terminology of \cite{Verbitsky}. Note that such bundle is rotable with respect to 
 all complex structures in $\cU=S^2$.
 \item If $\lambda\neq 0$, then $F_A$ is of type $(1,1)$ only with respect to $\pm I$,
and hence $E$ is not rotable.
\end{itemize}

\subsection{Hyperk\"ahler manifolds}
The previous case generalizes to higher dimensions. Let $M$ be a hyperk\"ahler 
manifold of dimension $4n$. This means that $M$ has a Riemannian metric whose
holonomy is $H=\Sp(n)<\SO(4n)$. The group $\Sp(n)$ is the group of endomorphisms
of $\RR^{4n}=\HH^n$ which commute with the quaternionic structure of $\HH^n$ as
an $\HH_R$-vector space (that is, $\HH$ acts on $\HH^n$ by multiplication on the right).

Therefore the elements of $\Sp(1)=S^3\subset \HH_R$, that is the quaternions of the form
$a i +bj+ck$, $a^2+b^2+c^2=1$,  produce endomorphisms $L=aI +bJ +c K$  on the tangent
space $TM$ which commute with the action of $H=\Sp(n)_L$, hence they are parallel with respect
to the Levi-Civita connection. This gives an $S^2$-family of complex structures on $M$
compatible with the Riemannian metric. 

Fix a complex structure $I$, given by some $U_I= \U(2n)$ with $\Sp(n)<U_I<\SO(4n)$.
This subgroup is the isotropy of $I$, 
which is $U_I=\Sp(n)\cdot S^1_I$, where $S^1_I=\{a\Id+\, b\, I\, | \, a^2+
b^2=1\}$. We have
 \begin{align*}
  N & = \Sp(n)\cdot \Sp(1), \\
  C & = U_I= \Sp(n) \cdot S^1_I. 
 \end{align*}
Hence 
 $$
 \cU'=N/C = \frac{\Sp(n)\cdot \Sp(1)}{\Sp(n)\cdot S^1_I} \cong \Sp(1)/ S^1_I \cong S^2.
 $$

The following result gives us the 
decomposition of the space of $2$-forms $\bigwedge^2$ under $\Sp(n)$. 
Consider the quaternionic space $V=\RR^{4n}=\HH^n$, with action of $\HH$ on the right.
The space $W=\bigwedge^2 V$ consists of bilinear antisymmetric maps $\varphi:V\x V\to \RR$.
Let $W_\HH$ be the subset of those
bilinear maps such that $\varphi(x I,y I)=\varphi(x J,y J)=\varphi(x K,y K)=\varphi(x,y)$, for 
all $x,y\in V$; let $W_I$ be the subset of those bilinear maps satisfying 
$\varphi(x I,y I)=-\varphi(x J,y J)=-\varphi(x K,y K)=\varphi(x,y)$, for 
all $x,y\in V$; define $W_J$ and $W_K$ similarly.
Finally note that $\omega_I\in W_I$ produces an (orthogonal)
decomposition $W_I=\la \omega_I\ra \oplus W_{I,prim}$. Then

\begin{lemma}\label{lem:eqn:xxx}
We have the following
  \begin{equation} \label{eqn:xxx} 
 \bigwedge\nolimits^2= \la \omega_I,\omega_J,\omega_K\ra \oplus W_\HH \oplus W_{I,prim} \oplus
W_{J,prim}\oplus W_{K,prim} \, .
 \end{equation}
With respect to the complex structure $I$, 
 \begin{align*}
 &\triangle^{1,1}_{I, prim} =W_\HH \oplus W_{I,prim}, \\
 &\triangle^{2,0}_I =\la \omega_J,\omega_K\ra \oplus W_{J,prim}\oplus W_{K,prim}.
 \end{align*}
and analogously for the other complex structures.
\end{lemma}

\begin{proof}
We have to see that $W=W_\HH \oplus W_I \oplus W_J\oplus W_K$.
First, note that $W=\bigwedge^2 V$ has dimension $\dim W=8n^2-2n$. 
Secondly, note that $W_\HH, W_I, W_J, W_K$ are complementary subspaces,
so their sum is a direct sum.

We introduce the following notation: for a quaternion $q=a+bi+cj+dk\in \HH$, let
$a=\Re(q)$, $b=\text{Im}(q)$, $c=\text{Jm}(q)$, $d=\text{Km}(q)$.
Take $A\in M_{n\x n}(\HH)$, and $\psi_A(x,y)=x^T A \, \overline{y}$.
Then for $A$ a real antisymmetric matrix, $\Re(\psi_A)\in W_\HH$, and for $A$ real symmetric, 
$\text{Im}(\psi_A)$, $\text{Jm}(\psi_A)$, $\text{Km}(\psi_A)\in W_\HH$. This implies
that $\dim W_\HH\geq 2n^2+n$.

On the other hand, for $A$ real antisymmetric, $\text{Im}(\psi_{Ai})\in W_I$, and
for $A$ real symmetric,  $\Re(\psi_{Ai})$, $\text{Jm}(\psi_{Ai})$, $\text{Km}(\psi_{Ai})\in W_I$.
Hence $\dim W_I\geq 2n^2-n$. 

Analogously, $\dim W_J\geq 2n^2-n$ and $\dim W_K\geq 2n^2-n$. So 
$W_\HH \oplus W_I \oplus W_J\oplus W_K$ has dimension at least $2n^2+n+ 3(2n^2-n)=8n^2-2n=
\dim W$. This proves that 
$W=W_\HH \oplus W_I \oplus W_J\oplus W_K$, and $\dim W_\HH=2n^2+n$,
$\dim W_I=\dim W_J=\dim W_K=2n^2-n$.
\end{proof}

Note that a hyperk\"ahler manifold
$(M,I)$ is holomorphically symplectic with symplectic form 
$\Omega_I=\omega_J+ \sqrt{-1} \omega_K\in 
\bigwedge^{2,0}_I$, and $\omega_J,\omega_K\in \triangle_I^{2,0}$.

The action of $\Sp(1)$ on the set of complex structures of 
$V$ acts on the decomposition (\ref{eqn:xxx}) by rotating the first space
and the last three summands. In particular,
 $$
 \cU' \cong S(\la \omega_I,\omega_J,\omega_K\ra)=S^2.
 $$

The main consequence of Lemma \ref{lem:eqn:xxx} is that 
 \begin{equation} \label{eqn:L}
  \triangle^{1,1}_{L}\cap \triangle^{1,1}_{L'}=W_\HH,
 \end{equation}
if $L,L'\in \cU'$ and $L'\neq \pm L$.

If $E\to M$ is a complex vector bundle with a connection $A$ which is HYM with respect to
$I$, then $F_A \in \bigwedge^{1,1}_I(\End E)$ and $\Lambda_I F_A =\lambda \Id$. 
By (\ref{eqn:L}), the conection $A$ is HYM with respect to some $L\neq \pm I$ if and only if 
 $$
  F_A \in W_\HH (\End E).
  $$ 
In the terminology of \cite{Verbitsky}, such bundles are called hyperholomorphic. We have 
thus the following definition. 

\begin{defn}\label{def:hyperh}
Let $E\to M$ be a complex vector bundle, and let $A$ be a connection which is HYM with respect to
 $I$. We say that $A$ is \emph{hyperholomorphic} if $F_A \in W_\HH (\End E)$.
 \end{defn}

Therefore, $E$ is a rotable bundle if and only if it is hyperholomorphic. In this case
$E$ is HYM with respect to all $L\in \cU'$. Note that,
in particular, it should be $\lambda =0$. 

We have a cohomological characterization of hyperholomorphic bundles as follows.

\begin{prop}[{\cite[Theorem 3.1]{Ve2}}]
Let $M$ be a compact hyperk\"ahler manifold.
and $E$ is a vector bundle HYM with respect to $I$. Then $E$ is
hyperholomorphic if and only if $c_1 (E), c_2 (E)$ are 
Hodge classes with respect to $J$ and $K$.
\end{prop}

Recall that a Hodge class with respect to some complex structure $L$ is a class
in $H^{p,p}_L(M)$, $p\geq 0$.

We have an alternative characterization of hyperholomorphic bundles in terms
of calibrations of the Chern classes.

\begin{thm} \label{thm:thm1}
 Let $M$ be a compact hyperk\"ahler manifold, and let $E$ be a vector bundle
 HYM with respect to $I$ with $\deg_I (E)=0$. Then 
  $$
   c_2(E)\cup [\omega_L]^{n-2} \leq c_2(E)\cup [\omega_I]^{n-2}
  $$
 for $L\in \cU'$ and $E$ is HYM with respect to $L$ if and only if there is equality.
\end{thm}

\begin{proof}
We have the following
 $$  
 \alpha \wedge \alpha \wedge \frac{1}{(n-2)!} \omega_I^{n-2} = 
  \left\{ \begin{array}{ll} 
  || \alpha||^2 \vol, \qquad & \alpha\in \triangle^{2,0}_I \\
  -|| \alpha||^2 \vol, \qquad & \alpha\in \triangle^{1,1}_{I,prim} \\
  (n-1) || \alpha||^2 \vol, \qquad & \alpha\in \la \omega_I \ra 
  \end{array} \right.
  $$
 Therefore
  \begin{align*}
  c_2(E)\cup \frac{1}{(n-2)!} [\omega_I]^{n-2} & = \frac{1}{8\pi^2} \int_M
  \Tr(F_A \wedge F_A) \wedge \frac{1}{(n-2)!} \omega_I^{n-2} \\
  &= \frac{1}{8\pi^2} (||F_A^{1,1,prim}||^2 -
   ||F_A^{2,0}||^2 - (n-1) || \Lambda_I F_A||^2 ) \\
   &= \frac{1}{8\pi^2} (||F_A||^2 - 
   3 ||F_A^{2,0}||^2 - n || \Lambda_I F_A||^2 ) ,
  \end{align*}
 using that $\la B,C\ra= -\Tr(B C)$ is the Killing metric in ${\mathfrak u}(r)$.

 Therefore  $c_2(E)\cup [\omega_L]^{n-2}$, $L\in \cU$, achieves its maximum
 if $F_A^{2,0}=0$ (w.r.t.\ $L$) and $\Lambda_L F_A=0$. In this case $A$ is HYM with respect to $L$.
 \end{proof}

Theorem \ref{thm:thm1} also appears as Claim 3.21 in \cite{Verbitsky} with a different proof.

\section{Complex tori}

\subsection{$\Spin$-rotation of complex $4$-tori}
Let $M=\RR^8/\Lambda$ be a real $8$-torus, where $\RR^8$ is endowed with the
standard Riemannian (flat) metric. Then the holonomy is trivial,
$H=\{1\} <\SO(8)$. We give $M$ the $\Spin(7)$-structure given by 
the standard $4$-form 
 \begin{align*}
 \Omega =& \,  dx_{1234}+ dx_{1256}+ dx_{1278} +dx_{1357} -dx_{1368}
   -dx_{1458}-dx_{1467} \\ &  -dx_{2358} -dx_{2367} -dx_{2457} +
   dx_{2468} + dx_{3456} +dx_{3478} +dx_{5678}
 \end{align*}
By definition, $G=\Spin(7)<\SO(8)$ is the isotropy subgroup of $\Omega$.

We consider the $\SU(4)$-structures compatible with the $\Spin(7)$-structure,
that is $U\cong \SU(4)$ with $U<G$. An $\SU(4)$-structure on $M$ is given
by a complex structure $I$, compatible with the metric, and a $(4,0)$-form
$\theta \in \bigwedge^{4,0}$ with $|\theta|=4$. The K\"ahler form is $\omega_I$.
The $\Spin(7)$-structure determined by $U$ is given by the $4$-form
$\Omega_U=\frac12 \omega_I^2 + \Re(\theta)$. We say that $U$ is compatible
with the given $\Spin(7)$-structure if $\Omega_U=\Omega$, or equivalently, if $U<G$.
The space $\cU$ is the space of all such $U$.

Fix an $\SU(4)$-structure $U_0=\SU(4)<\Spin(7)$ associated to $(I,\theta)$. Then 
 \begin{align*}
 N &=\Spin(7), \\
 C &=U_0=\SU(4).
\end{align*}
So the complex structures are parametrized by
 $$
 \cU'=N/C=\Spin(7)/\SU(4)\, . 
 $$
This space is a $6$-sphere. It is described in \cite[Lemma 1]{Mu} as follows.
The group $\Spin(7)$ acts on the $2$-forms, and the decomposition
in irreducible summands is $\bigwedge^2=\bigwedge^2_7\oplus \bigwedge^2_{21}$,
where $\bigwedge^2_7$ is a $7$-dimensional representation
and it consists of those $\alpha\in \bigwedge^2$ with $\Omega\wedge \alpha=3*\alpha$,
and $\bigwedge^2_{21}$ is a $21$-dimensional representation
and it consists of those $\alpha\in \bigwedge^2$ with $\Omega\wedge \alpha=-*\alpha$.
It is easy to see that $\omega_I\in \bigwedge^2_7$. Then the action of $\Spin(7)$
by conjugation on $U_0$ moves $\omega_I$ in $\bigwedge^2_7$ transitively in
the sphere $S(\bigwedge^2_7)$ of elements of norm $2$. That is,
 $$
 \Spin(7)/\SU(4) \cong S(\bigwedge\nolimits^2_7)=S^6 \, .
 $$

There is a map $\cL:\bigwedge^{2,0}\to\bigwedge^{0,2}$ given by
 $$
 \bigwedge\nolimits^{2,0}_I\cong (\bigwedge\nolimits^{2,0}_I)^* \cong
 (\overline{\bigwedge\nolimits^{0,2}_I})^* \cong \bigwedge\nolimits^{0,2}_I\, ,
 $$
where the first map is the duality given by $\theta$, the second map is
conjugation, and the third map is given by the hermitian metric. This
$\cL$ produces another map $\cL:\bigwedge^{0,2}_I\to\bigwedge^{2,0}_I$, and
by considering the real subspaces, a map $\cL:\triangle^{2,0}_I\to \triangle^{2,0}_I$.
It is easy to see that $\cL^2=\Id$, so there is a decomposition 
$\triangle^{2,0}_I=\triangle^{2,0}_{I,+}\oplus \triangle^{2,0}_{I,-}$ into two $6$-dimensional
subspaces, according to the eigenvalues of $\cL$. Then 
 \begin{align*}
 \bigwedge\nolimits^2_7 &=\triangle^{2,0}_{I,+}\oplus \langle \omega_I\rangle \\
 \bigwedge\nolimits^2_{21} &=\triangle^{2,0}_{I,-}\oplus \triangle^{1,1}_{I,prim} 
 \end{align*}
as it is computed in \cite[Proposition 2]{Mu} (see also \cite{DoTh}). The conclusion is
that given any $\gamma\in\triangle^{2,0}_{I,+}$, the form
 $$
 \omega= 2 \, \frac{\omega_I+\gamma}{|\omega_I+\gamma|}
 $$
defines another $\SU(4)$-structure in $\cU'$.

Let $E\to M$ be a hermitian complex vector bundle. Let $A$ be a hermitian 
connection which is HYM with respect to $I$.
Then $F_A \in \bigwedge^{1,1}_I(\End E)$ and $\Lambda_I F_A =\lambda \Id$. 
We decompose $F_A=F_A^o + \frac1r (\Tr F_A) \Id$, where $F_A^o$ is the trace-free
part. We have that $c_1(E)=[\frac{\imat}{2\pi} \Tr F_A]$ and 
 $$
  \beta(E)=c_2(E)- \frac{r-1}{2r} c_1(E)^2= \left[ \frac{1}{8\pi^2} \Tr(F_A^o \wedge F_A^o) \right] .
  $$

\begin{defn}
$A$ is a \textit{spinstanton} (a $\Spin(7)$-instanton
in the terminology of \cite{Lewis} or \cite{Mu}) if 
$F_A^o \in \bigwedge\nolimits^2_{21}(\End E)$.
\end{defn}

There is a cohomological criterium for Spin-rotation as follows

\begin{prop}
 Let $E\to M$ be a hermitian complex vector bundle. Let $A$ be a 
connection which is HYM with respect to $I$. Then $A$ is HYM with respect to
$L\in \cU'$ if and only if $c_1(E), c_2(E)$ are Hodge classes with respect to $L$.
\end{prop}

\begin{proof}
As $A$ is HYM with respect to $I$, we have that $F_A^o \in \bigwedge^{1,1}_I(\End E)$.
In particular, $F_A^o \in \bigwedge\nolimits^2_{21}(\End E)$ and $A$ is a spinstanton.
By \cite[Proposition 11]{Mu}, a spinstanton $A$ is traceless HYM with respect to $L$ (that is  
$F_A^o \in \bigwedge^{1,1}_L(\End E)$) if and only if $\beta(E)\in H^{2,2}_L(M)$. 

If $c_1(E)\in H^{1,1}_L(M)$, then $[\Tr F_A]$ is of type $(1,1)$, so $\Tr F_A= \beta+da$,
for some $\beta\in \Omega^{1,1}(M)$, and a $1$-form $a$. Changing the connection $A$ to 
$A+a \, \Id$, we have that $\Tr F_A \in \bigwedge^{1,1}$, and hence $F_A\in \bigwedge^{1,1}(\End E)$.
So $A$ is HYM with respect to $L$.
\end{proof}

There is also a characterization of Spin-rotability in terms of calibrations,
which is the analogue of Theorem \ref{thm:thm1} in this situation.

\begin{thm} \label{thm:Spin-rot}
  Let $(M,I)$ be a complex $4$-torus which is algebraic, and let
  $E\to M$ be a vector bundle which is HYM with respect to $I$. 
  Assume that $c_1(E)=0$. Then 
 $$
 c_2(E) \cup [\omega_L]^2 \leq c_2(E) \cup [\omega_I]^2
 $$
with equality if and only if $E$ is HYM with respect to $L$.
\end{thm}

\begin{proof}
Let us recall the result of \cite[Proposition 19]{Mu}. Consider 
 \begin{equation} \label{eqn:aaa2} 
  k= \frac{ \beta (E) \cup [\omega_I]^2}{ [\omega_I ]^4}\, .
 \end{equation}
Then 
 \begin{equation} \label{eqn:aaa}
  (\beta (E)-3k[\omega_I]^2)\cup [\gamma]^2 \leq 0,
 \end{equation}
for any $\gamma\in \triangle^{2,0}_{I,+}$. There is equality 
if and only if $E$ is traceless HYM with respect to $L$ with 
$\omega_L=2 \, \frac{\omega_I+\gamma}{|\omega_I+\gamma|}$.

Now let $\kappa^2=|\omega_I+\gamma|^2=4+|\gamma|^2$. So
 \begin{align*}
   \kappa^2 \, \beta(E)\cup [\omega_L]^2 &= 4\beta(E)\cup  [\omega_I+\gamma]^2\\ 
&= 4\beta(E)\cup \left( [\omega_I]^2 +
    2[\omega_I]\cup [\gamma] +[\gamma]^2 \right) \\
     & \leq 4\beta(E)\cup [\omega_I]^2 + 12 k [\omega_I]^2\cup [\gamma]^2 \\
     & = 4\beta(E)\cup [\omega_I]^2 + k |\gamma|^2 [\omega_I]^4 \\
     &=(4+|\gamma|^2) \beta(E)\cup [\omega_I]^2  \\ 
    &=  \kappa^2 \, \beta(E)\cup [\omega_I]^2,
 \end{align*}
using that $\beta(E)\cup [\omega_I]\cup [\gamma]=0$ in the second line,
(\ref{eqn:aaa}) in the third line, 
$[\omega_I]^2\cup [\gamma]^2=2|\gamma|^2 \frac{[\omega_I]^4}{4!}$
in the fourth line and the definition (\ref{eqn:aaa2}) of $k$ in the fifth line. Hence
 $$
 \beta(E) \cup [\omega_L]^2 \leq \beta (E) \cup [\omega_I]^2
 $$
with equality if and only if $E$ is traceless HYM with respect to $L$.
As $c_1(E)=0$, $\beta(E)=c_2(E)$ and $E$ is HYM with respect to $L$.
\end{proof}

This result determines a sphere $S^r \subset S^6$, where $0\leq r\leq 6$, 
(see \cite[Proposition 17]{Mu}), such that 
the bundle $E$ is rotable for the complex
structures in this sphere. The sphere $S^r$ can be
of different dimensions, depending on the bundle and manifold,
as the examples in \cite{Mu} show.

Moreover, there is an example in \cite{Mu} of
a complex torus $(M,\omega_I)$ and a rotable bundle $E\to (M,\omega_I)$
for which there is a rotated structure $L$ such that $(M,\omega_L)$ is,
as a complex torus, of very different nature: for instance
$(M,\omega_I)$ can be a decomposable complex abelian variety
and $(M,\omega_L)$ be an indecomposable complex abelian variety.

\subsection{Rotation of complex structures on tori}

For a $2n$-dimensional torus $M=\RR^{2n}/\Lambda$ (with a flat Riemannian metric),
we can consider the family of \emph{all} complex structures compatible with the metric.
This means that we take now $H=\{1 \}< G=\SO(2n)$. Let  $U_0=\U(n)< G$
be one complex structure $I$. Then
 \begin{align*}
 N &= \SO(2n), \\
 C &= \U(n).
 \end{align*}
The family of complex structures on $M$ is parametrized by
 $$
 \cU'=N/C=\SO(2n) /\U(n).
 $$

For a $4$-torus, $\cU'=\SO(4)/\U(2)\cong S^2$, and we recover the situation 
discussed previously for a hyperk\"ahler rotation. This is due to the fact that
a complex structure (a $\U(2)$-structure) determines uniquely an $\SU(2)$-structure.
So the rotations of complex structures for a $4$-torus are the same as the
ones obtained by considering it as hyperk\"ahler manifold. 
In \cite{Toma}, M. Toma has considered these rotations to construct new stable
bundles on complex $2$-tori.

For a $2n$-torus with $2n>4$, the situation is more complicated. For instance, for a
$6$-torus, the space 
 $$
 \cU'=\SO(6)/\U(3) \cong \CP^3\, .
 $$
This means that the orbit of $\omega\in \bigwedge^2$ under $\SO(6)$
is diffeomorphic to $\CP^3$. However, it is difficult to describe it explicitly,
since $\CP^3 \subset \bigwedge^2$ spans the whole of $\bigwedge^2$, as this is an irreducible $\SO(6)$-representation.
Moreover, if $E\to M$ is a vector bundle endowed with an HYM connection $A$ with respect to
$\omega$, then $F_A\in \bigwedge^{1,1}_I(\End E)$. For $A$ to be HYM with respect to some
other $L\in \cU'$, we need to check that $F_A \in\bigwedge^{1,1}_{L}(\End E)$. This is a 
condition to be checked at every point $p\in M$, giving a functional equation. 
In the case of $\Spin$-rotations for $8$-tori, the real 
power of Theorem \ref{thm:Spin-rot}
is that it gives a \emph{cohomological} condition for the functional equation 
$F_A \in\bigwedge^{1,1}_{L}(\End E)$ to hold everywhere.

If $E\to M$ is a bundle which is rotable for the whole family $\SO(2n)/U(n)$,
that is, which is HYM for all complex structures in the family $\SO(2n)/U(n)$, with
$n>2$, then $A$ is flat, i.e. $F_A=0$. This is shown in \cite{Ve3}.
Note that however, it is possible to have a bundle $E\to M$ which is rotable
for a subfamily $\cF\subset \SO(2n)/\U(n)$. For instance, take a $\Spin$-rotable
bundle (there are examples in \cite{Mu}) for a family $\cF\subset \Spin(7)/\SU(4)$.
Taking the image under the natural map
$\Spin(7)/\SU(4) \stackrel{\imath}{\too} \SO(8)/\U(4)$, we get a bundle
which is HYM for all complex structures in the family $\imath (\cF)\subset \SO(8)/\U(4)$.

\section{Product of two K3 surfaces}

Let $M, M'$ be two K3 surfaces. Then the holonomy of the manifold $X=M\x M'$ is
$H=\SU(2)\x \SU(2)< \SO(4)\x \SO(4) < G=\SO(8)$. 
Fix complex structures $I,I'$ on $M,M'$. This
determines groups $U_I=\U(2)<\SO(4)$, $U_{I'}
=\U(2)<\SO(4)$, and hence a subgroup $\U(2)\x \U(2) <\SO(4)$.
We have a unique 
 $$
 U=U_{\cI}=\U(4)<\SO(8)
 $$
given by the complex structure $\cI=I+I'$ on $X=M\x M'$.
Then 
 \begin{align*}
 N &=\SO(4)\x\SO(4), \\
 C &=\U(2)\x \U(2).
 \end{align*}
The quotient is
  $$
  \cU'=N/C =(\SO(4)/\U(2)) \x (\SO(4)/\U(2)) \cong S^2\x S^2.
  $$
If $I,J,K$ are the three complex structures of $M$ and 
$I',J',K'$ are the three complex structures of $M'$, then
$\cL=aI+bJ+cK + a'I'+b'J'+c'K'$, $(a,b,c), (a',b',c')\in S^2$,
is a complex structure in the family $\cU'$.

Write $\RR^8 =V\oplus V'$, where $V, V'\cong \RR^4$ correspond to the
two factors $M,M'$. Then there is a decomposition into five
irreducible components (under the group $N=\SO(4)\x \SO(4)$)
 \begin{equation}\label{eqn:dec}
 \bigwedge\nolimits^2 = \la \omega_I,\omega_J,\omega_K\ra \oplus 
 \triangle^{1,1}_{I,prim} V \oplus  \la \omega_{I'},\omega_{J'},\omega_{K'}\ra \oplus 
 \triangle^{1,1}_{I',prim} V' \oplus  D ,
 \end{equation}
where 
 $$
 D=\Re \left( \bigwedge\nolimits^{1,0}_I V\otimes \bigwedge\nolimits^{1,0}_{I'} V'\right) \oplus
 \Re \left( \bigwedge\nolimits^{1,0}_I V\otimes \bigwedge\nolimits^{0,1}_{I'} V' \right) .
 $$
Note that, for the complex structure $\cI=I+I'$, we have
 $$
  \triangle^{2,0}_{\cI}=
\la \omega_J,\omega_K\ra \oplus  \la \omega_{J'},\omega_{K'}\ra \oplus 
 \Re \left( \bigwedge\nolimits^{1,0}_I V\otimes \bigwedge\nolimits^{1,0}_{I'} V'\right).
  $$

\begin{lemma}\label{lem:ll}
Let $\alpha \in D$. For $(L,L') \in S^2\x S^2$, we have 
 $$
  - \int_X \alpha\wedge \alpha \wedge \omega_L \wedge \omega_{L'}  \leq 4 ||\alpha||^2\, ,
 $$
and equality holds if and only if $\alpha \in 
\Re \left( \bigwedge\nolimits^{1,0}_L V\otimes \bigwedge\nolimits^{0,1}_{L'} V' \right)$.
\end{lemma}

\begin{proof}
 Write $\alpha=\alpha_1+\alpha_2=\sum (a_i\wedge a_i'+\bar{a}_i\wedge \bar{a}'_i) + 
\sum (b_i\wedge \bar{b}'_i+\bar{b}_i\wedge {b}'_i)$, where 
 $a_i,b_i\in \bigwedge^{1,0}_I V $ and $a_i',b_i' \in  \bigwedge^{1,0}_{I'} V'$.
We have that $a_i\wedge \bar{a}_i \wedge \omega_L= -2\imat |a_i|^2 \vol_M$ and
 $a_i'\wedge \bar{a}'_i \wedge \omega_{L'}= -2\imat |a_i'|^2 \vol_{M'}$.
Then 
 \begin{align*}
  \int_X \alpha&\wedge \alpha \wedge \omega_L \wedge \omega_{L'} \\ &=
2 \sum \int_X a_i\wedge a_i'\wedge \bar{a}_j\wedge \bar{a}'_j \wedge \omega_L \wedge \omega_{L'}+
2 \sum \int_X b_i\wedge \bar{b}_i'\wedge \bar{b}_j\wedge {b}'_j \wedge \omega_L \wedge \omega_{L'} \\
&=
-2 \sum \int_X a_i\wedge \bar{a}_j\wedge a_i'\wedge \bar{a}'_j \wedge \omega_L \wedge \omega_{L'}+
2 \sum \int_X b_i\wedge \bar{b}_j\wedge {b}'_j \wedge \bar{b}_i'\wedge \omega_L \wedge \omega_{L'} \\
 &= 4\int_X |\alpha_1|^2 \vol_X -4 \int_X |\alpha_2|^2 \vol_X 
=  4 || \alpha_1||^2 - 4 ||\alpha_2||^2 .
\end{align*}
So $-\int_X \alpha \wedge \alpha \wedge \omega_L \wedge \omega_{L'}  =
4 || \alpha_2||^2 - 4 ||\alpha_1||^2 \leq 4||\alpha||^2$ and equality happens for $\alpha_1=0$.
The result follows. 
\end{proof}

Suppose that $E\to X$ is a complex vector bundle with a connection $A$ which is HYM with respect to 
${\cI}=I+ {I'}$. Then $F_A\in \bigwedge^{1,1}_{\cI}(\End E)$ 
and $\Lambda_{\cI}  F_A=\tilde\lambda \Id$.
Let 
 \begin{align*}
 c_1(E) &=a+a'\in H^2(X)=H^2(M)\oplus H^2(M'), \\
 c_2(E) &=b+b'+ y \in H^4(X)=H^4(M)\oplus H^4(M')\oplus (H^2(M)\ox H^2(M')).
 \end{align*}
Let also $\lambda=a\cup [\omega_I]/[\omega_I]^2$, $\lambda'=a'\cup [\omega_{I'}]/[\omega_{I'}]^2$,
so $\tilde \lambda=\frac{\lambda+\lambda'}{2}$. 
The following result tells us when $E$ is rotable.

\begin{thm} \label{thm:thm3}
$E$ is rotable only in the following cases:
\begin{itemize}
\item $y=0$, 
$\lambda=\lambda'=0$. The
rotations are given by the family $S^2\x S^2$.
\item $y=0$, $\lambda=0$, $\lambda'\neq 0$. The rotations are given 
by the family $S^2\x \{\pm I'\}$.
\item $y=0$, $\lambda \neq 0$, $\lambda'= 0$. The rotations are given 
by the family $\{\pm I\}\x S^2$.
\item $y\neq 0$, $\lambda=\lambda'=0$. Then $E$ is rotable for those $\cL=L+L'\in S^2\x S^2$ 
such that  
 $$
 c_2(E) \cup [\omega_L]\cup [\omega_{L'}]=
 c_2(E) \cup [\omega_I]\cup [\omega_{I'}].
 $$
This family is either an $S^2$ embedded diagonally in
$S^2\x S^2$, or else $E$ is not rotable.
\end{itemize}
\end{thm}

\begin{proof}
We decompose $F_A=F_1+F_2+F_3+F_4+F_5$ according to (\ref{eqn:dec}). Let 
$(L,L') \in S^2\x S^2$ be another complex structure. We have to see if $F_1, F_3$ and $F_5$
are of type $(1,1)$ with respect to $\cL=L+L'$.

We start by noticing that $F_2\wedge\omega_L=0$ and $F_4\wedge\omega_{L'}=0$ for any $(L,L')$.
Also $a\cup [\omega_L]=\frac{\imat}{2\pi} \int_M \Tr(F_1) \wedge \omega_L=
\frac{\imat}{2\pi} r \lambda [\omega_I]\cup [\omega_L]$, where $r=\rk (E)$.
Analogously, $a'\cup [\omega_{L'}]=
\frac{\imat}{2\pi} r \lambda' [\omega_{I'}]\cup [\omega_{L'}]$.
Then
 \begin{align*}
  c_2(E)\cup [\omega_L]\cup [\omega_{L'}] &=
  \frac{1}{8\pi^2} \int_X \Tr( F \wedge F) \wedge \omega_L \wedge \omega_{L'} \\
 &=  \frac{1}{8\pi^2} \int_X \Tr(F_5\wedge F_5) \wedge \omega_L \wedge \omega_{L'}
 +\frac{2}{8\pi^2} \int_X \Tr(F_1\wedge F_3 ) \wedge \omega_L \wedge \omega_{L'} \\
 &=
 -\frac{1}{8\pi^2} \int_X \la F_5\wedge F_5 \ra \wedge \omega_L \wedge \omega_{L'}
 -\frac{1}{4\pi^2} r \lambda \lambda' ([\omega_I]\cup [\omega_L])([\omega_{I'}]\cup [\omega_{L'}]),
\end{align*}
using that $\la A,B\ra = -\Tr(AB)$ is the Killing form on ${\frak u}(r)$, the Lie algebra of $\U(r)$.

Regarding the components $F_1,F_3$, we have that $F_1=\lambda \omega_I\Id$, $F_3=\lambda'\omega_{I'}\Id$.
If $\lambda,\lambda' \neq 0$, then $F_1,F_3$ are of type $(1,1)$ with respect to $\cL=L+L'$ only for the
complex structures $\pm I\pm I'$. Therefore $E$ is not rotable.

If $\lambda\lambda' =0$, then the formula above and
Lemma \ref{lem:ll} say that
$F_5 \in \Re \left( \bigwedge\nolimits^{1,0}_L V\otimes \bigwedge\nolimits^{0,1}_{L'} V' \right)(\End E)$
if and only if $c_2(E)\cup [\omega_L]\cup [\omega_{L'}]$ achieves its maximum.
Considering
 \begin{eqnarray*}
 \Psi: S^2 \x S^2 & \too & \RR \\
 (\omega_L,\omega'_L) & \mapsto & - \frac{1}{8\pi^2} \int_X \la F_5\wedge F_5\ra \wedge \omega_L \wedge \omega_{L'}\, , 
\end{eqnarray*}
the maximum is achieved for $(\omega_I,\omega_I')$, by assumption. Note that
$\Psi$ is bilinear (when considered as a functional on $\RR^3\x \RR^3$). It is easy to see that
we can choose an
orthonormal basis (that we shall call $\{I,J,K\}, \{I',J',K'\}$ again) in which $\Psi$ has matrix
 $$
 \left(\begin{array}{ccc} m_1 & 0 & 0 \\ 0 & m_2 & 0 \\ 0 & 0 & m_3 \end{array}\right),
 $$ 
with $m_1\geq m_2 \geq m_3$. 
If $m_1>m_2$ then $\Psi(\omega_L,\omega'_L)=m_1$ only for $\pm (I+I')$. If $m_1=m_2>m_3$ then
$\Psi(\omega_L,\omega_{L'})=m_1$ for $L=a I+bJ$, $L'=a I'+b J'$, for $a^2+b^2=1$.
Finally, if $m_1=m_2=m_3>0$ then 
$\Psi(\omega_L,\omega_{L'})=m_1$ for $L=a I+bJ+cK$, $L'=a I'+b J'+cK'$, for $a^2+b^2+c^2=1$.

\begin{rmk}
Note that $\Psi\neq 0$ if and only if $m_1\neq 0$. This is the same as to say $c_2(E)\cup [\omega_I]\cup
[\omega_{I'}]\neq 0$, i.e., $y\cup [\omega_I]\cup
[\omega_{I'}]\neq 0$. In particular, $y\neq 0 \iff 
y\cup [\omega_I]\cup
[\omega_{I'}]\neq 0$.
\end{rmk}

Now, if either $\lambda=0,\lambda'\neq 0$ or $\lambda \neq 0,\lambda'= 0$ then
looking at the components $F_1,F_3$, we have that $E$ is rotable 
only for $\cL=L \pm I'$, $L\in S^2$, in the first case, or
 $\cL=\pm I +L'$, $L'\in S^2$, in the second case.
But then looking at $F_5$, it must be $y=0$ (this implying that $F_5 \equiv 0$).

If $\lambda=\lambda'=0$, then $F_1, F_3=0$. So we only need to check that $F_5$ is of
type $(1,1)$ with respect to $\cL=L+L'$. By the discussion above this happens exactly when
 $$
  c_2(E)\cup [\omega_L]\cup [\omega_{L'}] = c_2(E)\cup [\omega_I]\cup [\omega_{I'}] .
 $$

Choose the basis $\{I,J,K\}$ and $\{I',J',K'\}$ as above. Then 
$E$ is rotable for
those $\cL=L+L'=a (I+I')+b(J+J')+c(K+K')$ such that 
$y\cup [\omega_L]\cup [\omega_{L'}] =y\cup [\omega_I]\cup [\omega_{I'}]$. 
As $F_5\in D (\End E)$ is of type $(1,1)$ with respect to $\cI=I+I'$, we have that 
  $$
 F_5\wedge F_5 \wedge (\omega_J+\imat \omega_K) \wedge (\omega_{J'}+\imat\omega_{K'})=0,
 $$
because $\omega_J+\imat \omega_K$ is of type $(2,0)$. This means that
$F_5\wedge F_5\wedge \omega_J\wedge\omega_{J'}=
F_5\wedge F_5\wedge \omega_K\wedge\omega_{K'}$, implying that $m_2=m_3$.
This means that either $E$ is not rotable, or $E$ is rotable by an $S^2$ family
embedded diagonally in $S^2\x S^2$.
\end{proof}

The rotability of $E$ can be expressed in terms of the
structure of holomorphic symplectic manifold. Recall that $\Omega_I= \omega_J+\imat \omega_K$,
$\Omega_{I'}= \omega_{J'}+\imat \omega_{K'}$ and $\Omega_{\cI}=\Omega_I+
\Omega_{I'}$. We have the following

\begin{cor}
Let $E$ be a hermitian vector bundle which is HYM with respect to $\cI=I+I'$. Suppose
that $\lambda=\lambda'=0$. Then $E$ is rotable if and only if 
 $$
 2 c_2(E) \cup [\omega_{\cI}]^2= c_2(E) \cup [\Omega_{\cI}]\cup [\overline{\Omega}_{\cI}].
 $$
\end{cor}

\begin{proof}
We have that 
  \begin{align*}
 F_A\wedge F_A \wedge\Omega_{\cI}\wedge \overline{\Omega}_{\cI} &=
  2 \Re(F_5 \wedge F_5 \wedge (\omega_J+\sqrt{-1}\omega_K) 
\wedge (\omega_{J'}- \sqrt{-1} \omega_{K'})) \\
 &=  2 F_5 \wedge F_5 \wedge \omega_J \wedge \omega_{J'} +
 2 F_5 \wedge F_5 \wedge \omega_K \wedge \omega_{K'} \, .
\end{align*}
So 
 $$
 c_2(E)\cup [\Omega_{\cI}]\cup [\overline{\Omega}_{\cI}]=m_2+m_3.
 $$
Then the condition of the statement is equivalent to $m_1=m_2=m_3$, which is equivalent
to rotability, by Theorem \ref{thm:thm3}.
\end{proof}

\begin{rmk}
Assume that $a,a'$ are primitive forms. Then 
we have a Bogomolov type inequality: $c_2(E) \cup [\omega_\cI]^2\geq 0$, and 
this is equal to $0$ if and only if $c_2(E) =b+b'$.
\end{rmk}

In \cite{Sc}, Schlickewei uses these rotations for the self-product of a K3
surface, $X=M\x M$, but considering only complex structures which are 
self-products of a complex structure on the K3 surface, that is, restricting
consideration to the diagonal
$\Delta \subset S^2\x S^2$. Such $X$ can be treated then as a hyperk\"ahler
manifold with the arguments of Section \ref{sec:1}.

\end{document}